\documentclass{amsart}

\usepackage{amssymb,color}
\usepackage{amsfonts}
\usepackage{amsmath}
\usepackage{euscript}
\usepackage{enumerate}
\usepackage{graphics}

\newtheorem{theorem}{Theorem}[section]

\newtheorem{note}[theorem]{Note}
\newtheorem{prop}[theorem]{Proposition}
\newtheorem{cor}[theorem]{Corollary}
\newtheorem{exa}[theorem]{Example}

\newtheorem*{Theorem1'}{Theorem 1'}

\theoremstyle{definition}

\theoremstyle{remark}

\numberwithin{equation}{section}

\newcommand \g{{\mathfrak g}}
\newcommand \h{{\mathfrak h}}
\newcommand \hn{{\mathfrak h}(n)}

\newcommand \gl{{\mathfrak {gl}}}
\renewcommand \sl{{\mathfrak {sl}}}

\newcommand \R{{\mathbb R}}
\newcommand \Co{{\mathbb C}}

\newcommand \GL{{\mathrm {GL}}}

\newcommand \al{{\alpha}}
\newcommand \be{{\beta}}
\newcommand \ga{{\gamma}}
\newcommand \dl{{\delta}}

\newcommand \im{{\mathrm {im}}}

\begin{document}

\title [modular representations of Heisenberg algebras]{modular representations of Heisenberg algebras}

\author{Fernando Szechtman}
\address{Department of Mathematics and Statistics, Univeristy of Regina, Canada}
\email{fernando.szechtman@gmail.com}
\thanks{The author was supported in part by an NSERC discovery grant}





\begin{abstract} Let $F$ be be an arbitrary field and
let $\h(n)$ be the Heisenberg algebra of dimension $2n+1$ over
$F$. It was shown by Burde that if $F$ has characteristic 0 then
the minimum dimension of a faithful $\hn$-module is $n+2$. We show
here that his result remains valid in prime characteristic $p$, as
long as $(p,n)\neq (2,1)$.

We construct, as well, various families of faithful irreducible
$\hn$-modules if $F$ has prime characteristic, and classify these
when $F$ is algebraically closed. Applications to matrix theory
are given.
\end{abstract}

\maketitle

\section{Introduction}

Let $F$ be an arbitrary field. For $n\geq 1$, let $\h(n,F)$, or
just $\hn$, stand for the Heisenberg algebra of dimension $2n+1$
over $F$. This is a 2-step nilpotent Lie algebra with
1-dimensional center. It was shown by Burde \cite{B} that when $F$
has characteristic~0 the minimum dimension of a faithful
$\hn$-module is $n+2$. Further results on low dimensional
imbeddings of nilpotent Lie algebras when $\mathrm{char}(F)=0$ can
be found in \cite{B}, \cite{CR}, \cite{BM}.

Here we extend Burde's result to arbitrary fields by showing that
the minimum dimension of a faithful $\hn$-module is always $n+2$,
except only when~$n=1$ and $\mathrm{char}(F)=2$. See \S\ref{sec3}
for details. Our main tool if $F$ has prime characteristic is the
classification of faithful irreducible $\hn$-modules when $F$ is
algebraically closed. As we were unable to find a proof of this in
the literature, one is included in \S\ref{sec2}. We construct, as
well, various families of faithful irreducible $\hn$-modules when
$F$ is an arbitrary field of prime characteristic (see
\S\ref{sec4} and \S\ref{sec5}) and furnish applications to matrix
theory, found in \S\ref{sec2}.

We fix throughout a symplectic basis
$x_1,\dots,x_n,y_1,\dots,y_n,z$ of $\hn$, i.e., one with
multiplication table $[x_i,y_i]=z$. Clearly, a
representation~$R:\hn\to\gl(V)$ is faithful if and only if
$R(z)\neq 0$. All representations will be finite dimensional,
unless otherwise mentioned. If $R:\g\to\gl(V)$ and $T:\g\to\gl(V)$
are representations of a Lie algebra $\g$, we refer to $T$ and $R$
as equivalent if there is $\Omega\in\mathrm{Aut}(\g)$ such that
$T$ is similar to $R\circ\Omega$.

\section{Faithful irreducible representations of $\hn$}\label{sec2}

\begin{prop}\label{modhn} Let $F[X_1,\dots,X_n]$ be the polynomial algebra
in $n$ commuting variables $X_1,\dots,X_n$ over $F$. For $q\in
F[X_1,\dots,X_n]$, let $m_q$ be the linear endomorphism
``multiplication by $q$" of $F[X_1,\dots,X_n]$. Let
$\al,\be_1,\dots,\be_n,\ga_1,\dots,\ga_n\in F$, where $\al\neq 0$.
Then

(1) $F[X_1,\dots,X_n]$ is a faithful $\mathfrak{h}(n)$-module via
$$
z\mapsto \al\cdot I, x_i\mapsto \beta_i\cdot
I+\alpha\cdot\partial/\partial X_i,\, y_i\mapsto \ga_i\cdot
I+m_{X_i}.
$$

(2) $F[X_1,\dots,X_n]$ is irreducible if and only if $F$ has
characteristic 0.

(3) Suppose $F$ has prime characteristic $p$. Then
$(X_1^p,\dots,X_n^p)$ is an $\mathfrak{h}(n)$-invariant subspace
of $F[X_1,\dots,X_n]$ and
$$V_{\al,\be_1,\dots,\be_n,\ga_1,\dots,\ga_n}=F[X_1,\dots,X_n]/(X_1^p,\dots,X_n^p)$$
is a faithful irreducible $\mathfrak{h}(n)$-module of dimension
$p^n$. Moreover, $V_{\al,\be_1,\dots,\be_n,\ga_1,\dots,\ga_n}$ is
isomorphic to $V_{\al',\be'_1,\dots,\be'_n,\ga'_1,\dots,\ga'_n}$
if and only if $\al=\al'$ and $\be_i=\be'_i,\ga_i=\ga'_i$ for
all~$1\leq i\leq n$. Furthermore,
$V_{\al,\be_1,\dots,\be_n,\ga_1,\dots,\ga_n}$ is equivalent to
$V_{1,0,\dots,0}$.
\end{prop}

\begin{proof} This is straightforward.
\end{proof}

\begin{theorem}\label{rephn} Suppose $F$ has prime characteristic
$p$. Let $R:\h(n)\to\gl(V)$ be a faithful irreducible
representation. Assume each $z,x_1,\dots,x_n,y_1,\dots y_n$ acts
on~$V$ with at least one eigenvalue in $F$, say
$\al,\be_1,\dots,\be_n,\ga_1,\dots,\ga_n\in F$, respectively (this
is automatic if $F$ is algebraically closed). Then $V$ is
isomorphic to $V_{\al,\be_1,\dots,\be_n,\ga_1,\dots,\ga_n}$.
\end{theorem}

\begin{proof} We divide the proof into various steps.

\smallskip

\noindent{\sc Step 1. }{\it $R(z)=\al\cdot I$, where $\al\neq 0$.}

\smallskip

By assumption $R(z)$ has an eigenvalue $\al\in F$. Let $U$ be the
$\al$-eigenspace of $R(z)$. Since $z\in Z(\hn)$, we see that $U$
is a non-zero $\hn$-invariant subspace of~$V$. But $V$ is
irreducible, so $U=V$, i.e., $R(z)=\al\cdot I$. Since $R$ is
faithful, $\al\neq 0$.

\smallskip

\noindent{\sc Step 2. }{\it For every $v\in V$ and every
$x,y\in\hn$ such that $[x,y]=z$, we have
$$
x^m yv=yx^m v+m\al x^{m-1}v,\quad m\geq 1.
$$
}
\indent This follows easily by induction by means of Step 1.

\smallskip

\noindent{\sc Step 3. }{\it Let $x\in\hn$ and suppose $\be\in F$
is an eigenvalue of $R(x)$. Let
$$
V(\be)=\{v\in V\,|\, (x-\be)^m v=0\text{ for some }m\geq 1\}.
$$
Then $V=V(\be)$.
}
\smallskip

By Step 1 we may assume that $x\not\in Z(\hn)$. Then there exists
$y\in \hn$ such that $[x,y]=z$.  Clearly $V(\be)$ is a subspace of
$V$, which is non-zero by assumption. We claim that it is
$\hn$-invariant. Since $\hn=C_{\hn}(x)\oplus F\cdot y$, it
suffices to show that $V(\be)$ is $y$-invariant. Let $v\in
V(\be)$. Then $(x-\be)^m v=0$ for some $m\geq 1$. Since
$[x-\be/\al\cdot z,y]=0$, Steps 1 and 2 give
$$
(x-\be)^{m+1} yv=y (x-\be)^n v +m\al (x-\be)^m v=0.
$$
As $V$ is irreducible, we deduce $V=V(\be)$.

\smallskip

\noindent{\sc Step 4. }{\it Suppose $x,y\in\hn$ satisfy $[x,y]=z$
and that $w\in V$ is an eigenvector of $R(x)$ with eigenvalue
$\be$. Let $U$ be the $F$-span of all $y^i w$, $i\geq 0$. Then $U$
is invariant under $x$, and has basis $w,yw,\dots,y^{pm-1}v$ for
some $m\geq 1$. The matrix of $R(x)|_U$ relative to this basis is
the direct sum of $m$ copies of $M_{\al,\be}\in\gl(p)$, defined by
\begin{equation}
\label{defrep} M_{\al,\be}=\left(\begin{array}{ccccc}
  \be & \al  & 0 & \dots & 0 \\
  0 & \be & 2\al & \dots & 0 \\
  \vdots & \vdots & \ddots & \ddots & \vdots \\
  0 & 0 & \dots & \be & (p-1)\al \\
  0 & 0 & \dots & \dots & \be \\
\end{array}%
\right).
\end{equation}
In particular, the minimal polynomial of $x$ acting on $U$ is
$(X-\beta)^p$. }

\smallskip

Since $[y,-x]=z$, Step 2 gives
\begin{equation}
\label{kk} x y^m w=\beta y^m w+m\al y^{m-1}w,\quad m\geq 1.
\end{equation}
It follows from (\ref{kk}) that $U$ is invariant under the
Heisenberg subalgebra $\langle x,y,z\rangle$ of $\hn$. In
particular, $z$ acts with trace 0 on $U$. Since $z$ acts through
$\al$, with $\al\neq 0$, we must have $p|\dim U$. On the other
hand, $U$ is the $y$-invariant subspace of $V$ generated by $w$,
so it has basis $w,yw,\dots,y^{d-1}w$, where $d$ is the first
positive exponent such that $y^d w$ is a linear combination of
$w,y w,\dots,y^{d-1}w$. Since $p|\dim U$, we must have $d=pm$ for
some $m\geq 1$. It is now clear from (\ref{kk}) that the matrix of
$x$ acting on $U$ relative to the basis $w,y w,\dots,y^{pm-1}w$ is
the direct sum of $m$ copies of $M_{\al,\be}$.

\noindent{\sc Step 5. }{\it There is a common eigenvector $v\in V$
for the action of $z,x_1,\dots,x_n$; the $F$-span of all
$y^{i_1}\cdots y_n^{i_n}v$ such that $0\leq i_1,\dots,i_n<p$ is
$V$; the minimal polynomials of $R(x_k)$ and $R(y_k)$ are
$(X-\be_k)^p$ and $(X-\ga_k)^p$, respectively, for every $1\leq
k\leq n$. }
\smallskip

By hypothesis and Step 3 each $R(z),R(x_1),\dots,R(x_n)$ is
triangularizable. Since they commute pairwise, they are
simultaneously triangularizable. In particular, there is a common
eigenvector $v\in V$ for the action of $z,x_1,\dots,x_n$.

Let $W$ be the $F$-span of all $y^{i_1}\cdots y_n^{i_n}v$, where
$i_1,\dots,i_n\geq 0$. It follows from (\ref{kk}) that $W$ is
$\hn$-invariant. Since $W$ is non-zero, we deduce $W=V$.

Fix $1\leq k\leq n$. Given any sequence of non-negative integers
$i_1,\dots,\widehat{i_k},\dots,i_n$, consider the vector
$w=y_1^{i_1}\cdots \widehat{y_k^{i_k}}\cdots y_n^{i_n}v$, where
the symbol under the hat is to be omitted. Let $U_{w}$ be the
subspace of $V$ spanned by all $y_k^{i_k}w$, $i_k\geq 0$. By
Step~4, $U_w$ is $x$-invariant and either $U_w=0$ or $x$ acts on
$U_w$ with minimal polynomial $(X-\be_k)^p$. Since $U_v$ is
non-zero and $V$ is the sum of all $U_w$, the minimal polynomial
of $R(x_k)$ is $(X-\be_k)^p$.

Since $y_1,\dots,y_n,-x_1,\dots,-x_n,z$ is also a symplectic basis
of $V$, we deduce from above that every $R(y_k)$ has minimal
polynomial $(X-\ga_k)^p$.

It follows that $V$ is spanned by all $y^{i_1}\cdots y_n^{i_n}v$
such that $0\leq i_1,\dots,i_n<p$.

\smallskip

\noindent{\sc Step 6. }{\it The vectors $y^{i_1}\cdots
y_n^{i_n}v$, where $0\leq i_1,\dots,i_n<p$, form a basis of $V$. }
\smallskip

We argue by induction on $n$. The case $n=1$ was proven in Step 4.
Suppose $n>1$ and the result is true for $n-1$. Let $S$ be the
$F$-span of all $y^{i_1}\cdots y_{n-1}^{i_{n-1}}v$ such that
$0\leq i_1,\dots,i_{n-1}<p$. By Step 5, we have
\begin{equation}
\label{SyV} S+y_n S+\cdots+y_n^{p-1} S=V.
\end{equation}
We claim that this sum is direct. Indeed, let
$s_0,\dots,s_{p-1}\in S$ and assume
$$
s_0+y_n s_1+\cdots+y_n^{p-1} s_{p-1}=0.
$$
Suppose, if possible, that not all $y^j s_j$ are 0, and choose $j$
as large as possible subject to  $y^j s_j\neq 0$. Using
(\ref{kk}), we obtain
$$
0=(x_n-\be_n)^{j}(s_0+y_n s_1+\cdots+y_n^{p-1} s_{p-1})=j! \al^{j}
s_j\neq 0,
$$
a contradiction.

Now $S$ is a non-zero $\h(n-1)$-submodule of $V$. Suppose, if
possible, that $S$ is reducible and let $T$ be a non-zero proper
$\h(n-1)$-submodule of $S$. Then
$$
T\oplus y_n T\oplus \cdots\oplus y_n^{p-1} T
$$
is  a non-zero proper $\hn$-submodule of $V$, a contradiction.
Therefore $S$ is irreducible. This and Step 3 allow us to apply
the inductive hypothesis to obtain that all $y^{i_1}\cdots
y_n^{i_{n-1}}v$, such that $0\leq i_1,\dots,i_{n-1}<p$, are
linearly independent. This, the fact that the sum (\ref{SyV}) is
direct and Step 4 imply that all vectors $y^{i_1}\cdots
y_n^{i_n}v$, such that $0\leq i_1,\dots,i_n<p$, are linearly
independent.

\smallskip

\noindent{\sc Step 7. }{\it $V$ is isomorphic to
$V_{\al,\be_1,\dots,\be_n,\ga_1,\dots,\be_n}$.}

\smallskip

Since $x_1,\dots,x_n,y_1-\ga_1/\al\cdot z,\dots,y_n-\ga_n/\al\cdot
z,z$ is also a symplectic basis of $\hn$, it follows from Step 6
that all $(y_1-\ga_1)^{i_1}\cdots (y_n-\ga_n)^{i_n}v$, such that
$0\leq i_1,\dots,i_{n}<p$, also form a basis of $V$. We easily
verify that the action of $x_1,\dots,x_n,y_1,\dots,y_n,z$ on this
basis is the same as the action of these elements on the basis of
 $V_{\al,\be_1,\dots,\be_n,\ga_1,\dots,\ga_n}$ associated to
all $X_1^{i_1}\cdots X_n^{i_n}$, where $0\leq i_1,\dots,i_{n}<p$.
\end{proof}

\begin{cor} \label{co0} Suppose $F$ has prime characteristic $p$. Let
$R:\hn\to\gl(V)$ be a faithful irreducible representation. Then
$\dim V=p^n\times m$ for some $m\geq 1$.
\end{cor}

\begin{proof} Let $f\in F[X]$ be the minimal polynomial of $R(z)$.
Since $z$ is central in $\hn$ and $R$ is irreducible, we see that
$f$ is irreducible. Moreover, since $R$ is faithful, 0 is not a
root of $f$. Let $K$ be an algebraic closure of $F$. Then
$V_K=V\otimes K$ is a faithful representation of $\h(n,K)$. Let
$\al_1,\dots,\al_r$ be the roots of $f$ in $K$. Let $V_K(\al_i)$
be the generalized $\al_i$-eigenspace of $z$ acting on $V_K$. Then
$$
V_K=V_K(\al_1)\oplus\cdots\oplus V_K(\al_r),
$$
where each $V_K(\al_i)$ is an $\h(n,K)$-submodule of $V_K$.
Combining compositions series for each individual $V_K(\al_i)$
produces a composition series for $V_K$, where all composition
factors are faithful. This and Theorem \ref{rephn} yield the
desired result.
\end{proof}

\begin{cor}\label{co1} Suppose $F$ has prime characteristic $p$. Let
$R:\hn\to\gl(V)$ be a faithful representation of dimension $p^n$.
Then

(1) $R$ is irreducible and $R(z)=\al\cdot I$, where $\al\in F$ is
non-zero.

(2) The minimal polynomial of each $R(x_k)$ (resp. $R(y_i)$) is of
the form $X^p-\delta_k$ (resp. $X^p-\varepsilon_k$), where
$\delta_k\in F$ (resp. $\varepsilon_k\in F$).

(3) If $R':\hn\to \gl(V')$ is also faithful representation of
dimension $p^n$ then $V'\cong V$ if and only if
\begin{equation}
\label{nue} \al'=\al,\,\dl_k=\dl'_k,\,
\varepsilon'_k=\varepsilon_k,\quad 1\leq k\leq n.
\end{equation}

(4) If $F$ is perfect then $V\cong
V_{\al,\be_1,\dots,\be_n,\ga_1,\dots,\ga_n}$, where
$\be_1,\dots,\be_n,\ga_1,\dots,\ga_n\in F$ are the eigenvalues of
$x_1,\dots,x_n,y_1,\dots,y_n$ acting on $V$.
\end{cor}

\begin{proof} (1) Let $K$ be an algebraic closure of $F$. Then
$V\otimes K$ is a faithful $\h(n,K)$-module of dimension $p^n$. By
Theorem \ref{rephn}, $V\otimes K$ is an irreducible
$\h(n,K)$-module, whence $V$ is an irreducible $\hn$-module.

Since $V\otimes K$ is irreducible, the minimal polynomial of $z$
acting on $V\otimes K$ has degree 1. This is the same as the
minimal polynomial of $z$ acting on $V$. It follows that
$R(z)=\al\cdot I$, where $\al\in F$ is non-zero.

(2) As seen in the proof of Theorem \ref{rephn}, each $x_k$ (resp.
$y_k$) acts on $V\otimes K$ with minimal polynomial
$(X-\be_k)^p=X^p-\be_k^p$ (resp. $(X-\ga_k)^p=X^p-\ga_k^p$), where
$\be_k\in K$ (resp. $\ga_k\in K$). This is the same as the minimal
polynomial of $x_k$ (resp. $y_k$) acting on $V$, so $\be_k^p\in F$
(resp. $\ga_k^p\in F$).

(3) Clearly, if $V\cong V'$ then (\ref{nue}) is true. Suppose,
conversely, that (\ref{nue}) holds. It follows that every element
of the symplectic basis $x_1,\dots,x_n,y_1,\dots,y_n,z$ acts with
same eigenvalues on $V\otimes K$ and $V'\otimes K$, whence
$V\otimes K\cong V'\otimes K$, by Theorem~\ref{rephn}. It follows
from \cite{CR2}, \S 29, applied to the universal enveloping
algebra of $\hn$, that $V\cong V'$.

(4) If $F$ is perfect then $\be_k\in F$ (resp. $\ga_k\in F$),
whence $V\cong V_{\al,\be_1,\dots,\be_n,\ga_1,\dots,\ga_n}$, by
Theorem \ref{rephn}
\end{proof}

\begin{cor}\label{co2} Suppose $F$ has prime characteristic $p$. Let $A,B,C
\in\gl(p)$ be any matrices satisfying
\begin{equation}
\label{nue2} [A,B]=C\neq 0,\, [A,C]=0=[B,C].
\end{equation}
Then

(1) The only subspaces of the column space $V=F^p$ invariant under
$A$ and $B$ are $0$ and $V$.

(2) $C=\al\cdot I$, where $\al\in F$ is non-zero.

(3) $A$ (resp. $B$) is similar to the companion matrix of the
polynomial $X^p-|A|$ (resp. $X^p-|B|$).

(4) Suppose $A',B',C' \in\gl(p)$ also satisfy (\ref{nue2}). Then
there exists $X\in\GL(p,F)$ such that
$$
X^{-1}AX=A',\, X^{-1}BX=B',\, X^{-1}CX=C'
$$
if and only if
$$
|A|=|A'|,\, |B|=|B'|,\, |C|=|C'|.
$$

(5) If $F$ is perfect then $\be=|A|^{1/p}\in F$, $\ga=|B|^{1/p}\in
F$ and there exists $X\in\GL(p,F)$ such that
$$
A'=X^{-1}AX\text{ and }B'=X^{-1}BX
$$
satisfy
$$
A'=\left(\begin{array}{ccccc}
  \be & \al  & 0 & \dots & 0 \\
  0 & \be & 2\al & \dots & 0 \\
  \vdots & \vdots & \ddots & \ddots & \vdots \\
  0 & 0 & \dots & \be & (p-1)\al \\
  0 & 0 & \dots & \dots & \be \\
\end{array}%
\right),\, B'=\left(\begin{array}{cccc}
  \ga & 0   & \dots & 0 \\
  1 & \ga & \dots & 0 \\
  \vdots & \ddots & \ddots & \vdots \\
  0 &  \dots & 1 & \ga \\
\end{array}%
\right).
$$
\end{cor}

\begin{proof} This follows from Corollary \ref{co1}.
\end{proof}

\begin{cor}\label{co3} Suppose $F$ has prime characteristic $p$ and let
$\al,\dl_1,\dots,\dl_{p-1}\in F$, where $\al\neq 0$. Then the
matrix $D\in\gl(p)$, defined by
\begin{equation}
D=\label{mezx}
\left(%
\begin{array}{cccccc}
  0 & \al & 0 & \dots & \dots & 0 \\
  \dl_1 & 0 & 2\al & \dots &  \dots & 0 \\
  \dl_2 & \dl_1 & 0 & 3\al & \dots & 0 \\
  \vdots & \ddots & \ddots & \ddots & \ddots & \vdots \\
   \dl_{p-2} & \dots & \dl_2 & \dl_1 & 0 & (p-1)\al \\
  \dl_{p-1} & \dl_{p-2} & \dots & \dl_2 & \dl_1 & 0 \\
\end{array}%
\right)
\end{equation}
is similar to the companion matrix of the polynomial $X^p-|D|$. In
particular, if $F$ is perfect, then $D$ is similar to the Jordan
block $J_p(|D|^{1/p})$.
\end{cor}

\begin{proof} Let $A'$ and $B'$ be as defined in Corollary
\ref{co2} and let $C=\al\cdot I$. There is a polynomial $f\in
F[X]$ such that $D=A'+f(B')$. Since $[A',B']=C$, we infer
$$
[D,B']=C,\, [D,C]=0=[B',C],
$$
so Corollary \ref{co2} applies.
\end{proof}

\begin{exa}\label{eq} Take $\dl_1=1$ and $\dl_i=0$ if $2\leq i<p$ in
Corollary \ref{co3}. If $p=2$ then $D$ is the companion matrix of
$X^2-\al$, while if $p>2$ then row 1 of $D$ is a linear
combination of the remaining odd numbered rows, so $D$ is
nilpotent.
\end{exa}

\begin{note}{\rm Suppose the hypotheses of Theorem \ref{rephn} are
met and let $\g=\hn$. It is not true, in general, that every
$x\in\g\setminus Z(\g)$ acts on $V$ with minimal polynomial
$(X-\be)^p$ for some $\be\in F$. But it is almost always true. In
fact, the only exception occurs when $p=2$, $\al\notin F^2$ and
$x=s+t$, where $s$ and $t$ are in $F$-span of $x_1,\dots,x_n$ and
$y_1,\dots,y_n$, respectively, and $[s,t]\neq 0$.

Indeed, let $x\in\g\setminus Z(\g)$ and let $K$ be an algebraic
closure of $F$. Consider the Lie algebra $\g_K=\g\otimes K$ over
$K$. It follows from Theorem~\ref{rephn} (or Corollary \ref{co1})
that $V_K=V\otimes K$ is an irreducible $\g_K$-module. The minimal
polynomials of $x$ acting on $V$ and $V_K$ are the same. Call this
common polynomial $f\in F[X]$. Since $x$ belongs to a symplectic
basis of $\g_K$, Theorem \ref{rephn} implies that $f=(X-\be)^p$,
where $\be\in K$. We need to decide when $f$ has a root in $F$.
Now $x=s+t$, where $s$ is in the $F$-span of $x_1,\dots,x_n$ and
$t$ in the $F$-span of $y_1,\dots,y_n$. The hypotheses of
Theorem~\ref{rephn} ensure that both $s$ and $t$ have eigenvalues
in $F$, say $\gamma$ and $\delta$, respectively. If $[s,t]=0$ it
easily follows that $R(x)$ has a root in $F$. Scaling $s$, if
necessary, by a non-zero element of $F$, we may assume without
loss of generality that $[s,t]=z$. Let $s'=s-\ga/\al\cdot z$,
$t'=t-\dl/\al\cdot z$ and $x'=s'+t'$. Since $R(x)$ and $R(x')$
differ by a scalar operator -the scalar being in $F$-, we may
replace $R(x)$ by $R(x')$ without loss. Now both $R(s')$ and
$R(t')$ have minimal polynomial $X^p$. Let $w$ be an eigenvector
for $s'$. Then $w,t'w,\dots,(t')^{p-1}w$ form a basis for a
subspace $U$ of $V$ that is invariant under $\langle
s',t',z\rangle$. Relative to this basis, the matrix of $R(x')|_U$
is the matrix $D$ considered in Example \ref{eq}. Thus, if $p>2$
then $D$ is nilpotent and $R(x')$ has a root in $F$. If $p=2$ then
the minimal polynomial of $R(x')$ is divisible by, and hence equal
to, $X^2-\al$. This has a root in $F$ if and only if $\al\in F^2$.
}
\end{note}

\begin{note}{\rm The condition that $F$ be perfect is essential in
Corollary \ref{co1}. Indeed, suppose $F$ is imperfect and let
$\gamma\notin F^p$. Then $X^p-\ga\in F[X]$ is irreducible. Let
$\al\in F$ be non-zero, let $C\in\gl(p)$ be the companion matrix
of $X^p-\ga$, and let $M_{\al,\be}$ be defined as in
(\ref{defrep}).  Then
$$
x_1\mapsto M_{\al,\be},\, y_1\mapsto C, z\mapsto \al\cdot I
$$
defines a faithful irreducible representation of $\h(1)$ of
dimension $p$, through which $y_1$ acts with no eigenvalues from
$F$. }
\end{note}

\section{Faithful representations of minimum degree}\label{sec3}

For the proof of the following result we rely on \cite{B} as well
as on the classification of irreducible $\h(n)$-modules when $F$
is algebraically closed of prime characteristic, as given in
Theorem \ref{rephn}

\begin{theorem} Let $F$ be any field and let $n\geq 1$. Let $\h(n)$ be the
Hesienberg algebra of dimension $2n+1$ over $F$ and let $d(n)$ be
the minimum dimension of a faithful $\h(n)$-module. Then
$d(n)=n+2$, except only that $d(1)=2$ when $\mathrm{char}(F)=2$.
\end{theorem}

\begin{proof} Since $\h(1)\cong \sl(2)$ when $F$ has characteristic
2, it is clear that $d(1)=2$ in this case. Suppose henceforth that
$(n,\mathrm{char}(F))\neq (1,2)$.

The existence of a faithful $\h(n)$-module of dimension $n+2$ is
well-known, so $d(n)\leq n+2$.

Let $R:\h(n)\to\gl(V)$ be faithful module. We wish to show
$\dim(V)\geq n+2$. Since $R$ is faithful, we have $R(z)\neq 0$, so
there is $v\in V$ such that $R(z)v\neq 0$. Consider the linear map
$T:\h(n)\to V$ given by $T(x)=R(x)v$. Let $A=\ker T$ and $B=\im\,
T$. Clearly $A$ is a subalgebra of $\h(n)$. Since
$[\h(n),\h(n)]=F\cdot z$ and $R(z)v\neq 0$, it follows that $A$ is
abelian and $z\not\in A$, whence $\dim A\leq n$, and a fortiori
$\dim B\geq n+1$. If $\dim(B)\geq n+2$, we are done. Suppose
henceforth that $\dim B=n+1$. Then $A\oplus F\cdot z$ is a maximal
abelian subalgebra of $\hn$.

\noindent{\sc Case 1.} $R(z)$ is nilpotent. Suppose, if possible,
that $v\in B$. Then $R(x)v=v$ for some $x\in\hn$. Then $x\not\in
A\oplus F\cdot z$ by the definition of $A$ and the nilpotency
of~$R(z)$. By the maximality of $A$ and the fact that
$[\h(n),\h(n)]=F\cdot z$, there is $y\in A$ such that $[x,y]=z$.
Therefore,
$$
R(z)v=R(x)R(y)v-R(y)R(x)v=0,
$$
a contradiction. Thus, in this case $v\not\in B$, whence $\dim
V\geq \dim B+1=n+2$.

\noindent{\sc Case 2.} $R(z)$ is not nilpotent. Let $K$ be an
algebraic closure of $F$. Then $V_K=V\otimes K$ is an
$\h(n,K)$-module. By assumption, $V_K$ must have at least one
faithful composition factor $W$. By Schur's Lemma, $z$ acts as a
scalar operator on $W$. But $z$ acts with trace 0 on $W$, so $F$
has prime characteristic $p$. Hence, by Theorem \ref{rephn}, we
have $\dim_K W=p^n$. On the other hand $(n,p)\neq (1,2)$, so
$p^n\geq n+2$. Therefore $$\dim_F V=\dim_K V_K\geq \dim_K W\geq
p^n\geq n+2.$$
\end{proof}

\section{Irreducible representations obtained by restriction}\label{sec4}

\begin{prop} Let $K=F[\al]$ be a finite field extension of $F$.
Let $\g$ be a Lie algebra over $F$ and let $\g_K=\g\otimes K$ be
its extension to $K$. Let $R:\g_K\to\gl(V)$ be an irreducible
representation, possibly infinite dimensional. Suppose the
following condition holds:

(C) $\al\cdot I$ is in the associative $F$-subalgebra of
$\mathrm{End}(V)$ generated by $R(\g)$.

\noindent Then $V$ is an irreducible $\g$-module.
\end{prop}

\begin{proof} Let $U$ be a non-zero $\g$-submodule of $V$. Then
$KU$ is a non-zero $\g_K$-invariant $K$-subspace of $V$, so
$KU=V$. Since $K=F[\alpha]$, it follows from (C) that $ku\in U$
for all $u\in U$ and $k\in K$, whence $KU\subseteq U$, so $U=V$.
\end{proof}

\begin{note}{\rm Condition (C) cannot be dropped
entirely. Indeed, the natural module for $\sl(2,\Co)$ is
irreducible, but not as $\sl(2,\R)$-module. }
\end{note}

\begin{cor}\label{rds} Suppose that $F$ has prime characteristic $p$ and
let $\alpha$ be an algebraic element of degree $m$ over $F$ and
set $K=F[\alpha]$. Let $f_1,\dots,f_n,g_1,\dots,g_n\in F[X]$. Then
the irreducible $\h(n,K)$-module
$V_{\al,f_1(\al),\dots,f_n(\al),g_1(\al),\dots,g_n(\al)}$ of
dimension $p^n$ restricts to a faithful irreducible
$\h(n,F)$-module of dimension $p^n\times m$.
\end{cor}

\begin{note}{\rm  Here we furnish a matrix version of
Corollary \ref{rds} when $n=1$.

Let $F$ have prime characteristic $p$. Let $m\geq 1$ and suppose
$q\in F[X]$ is a monic irreducible polynomial of degree $m$ with
companion matrix $C$. Let $f,g\in F[X]$ and consider $A,B,D\in
\gl(pm)$ defined in terms of $m\times m$ blocks as follows:
$$
 A=\left(\begin{array}{ccccc}
  f(C) & C  & 0 & \dots & 0 \\
  0 & f(C) & 2C & \dots & 0 \\
  \vdots & \vdots & \ddots & \ddots & \vdots \\
  0 & 0 & \dots & f(C) & (p-1)C \\
  0 & 0 & \dots & \dots & f(C) \\
\end{array}%
\right),\, B=\left(\begin{array}{cccc}
  g(C) & 0   & \dots & 0 \\
  I & g(C) & \dots & 0 \\
  \vdots & \ddots & \ddots & \vdots \\
  0 &  \dots & I & g(C) \\
\end{array}%
\right),
$$
$$
D=\left(%
\begin{array}{cccc}
  C & 0 & \dots & 0 \\
  0 & C & \dots & 0 \\
  \vdots & \vdots & \ddots & \vdots \\
  0 & 0 & \dots & C \\
\end{array}%
\right).
$$
Then
$$
x_1\mapsto A,\, y_1\mapsto B, z\mapsto D
$$
defines a faithful irreducible $\h(1)$-module of dimension
$p\times m$. }
\end{note}

\section{Irreducible representations obtained from companion matrices}\label{sec5}

Here we produce faithful irreducible representations of $\hn$ not
equivalent to any obtained earlier, as long as $F$ is not
algebraically closed. Recall that a module is said to be uniserial
if its submodules form a chain.

\begin{theorem}\label{cuando} Suppose $F$ has prime characteristic $p$. Let $\al,\be\in F$, with $\al\neq 0$.
Let $M=M_{\al,\be}\in\gl(p)$ be defined as in (\ref{defrep}). Let
$m\geq 1$ and let $f(X)\in F[X]$ be an arbitrary monic polynomial
of degree $m$. Let $C\in \gl(pm)$ be the companion matrix of
$f(X^p)$. Then

(1) The linear map $R:\h(1)\to\gl(pm)$ given by
$$ x_1\mapsto
M\oplus\cdots\oplus M,\, y_1\mapsto C,\, z\mapsto \alpha\cdot I
$$
defines a faithful representation of $\h(1)$, where $x_1$ acts
with minimal polynomial $(X-\be)^p$ and $y_1$ with minimal
polynomial $f(X^p)$.

(2) If $f$ is irreducible then $R$ is
irreducible.

(3) If $f(X)=(X-\ga^p)^m$ for some $\ga\in F$, then $R$ is
uniserial with $m$ composition factors, all isomorphic to
$V_{\al,\be,\ga}$.
\end{theorem}

\begin{proof} (1) Consider the matrices $N,J\in\gl(p)$ defined as follows:
$$
N=\left(\begin{array}{cccc}
  0 & 0   & \dots & 0 \\
  1 & 0 & \dots & 0 \\
  \vdots & \ddots & \ddots & \vdots \\
  0 &  \dots & 1 & 0 \\
\end{array}\right),\, J=\left(\begin{array}{cccc}
  0 & 0   & \dots & 1 \\
  0 & 0 & \dots & 0 \\
  \vdots & \ddots & \ddots & \vdots \\
  0 &  \dots & 0 & 0 \\
\end{array}\right).
$$
We further let $K_i=k_i J\in\gl(p)$, $0\leq i\leq m-1$, where
$k_i\in F$. Then
$$
y_1\mapsto\left(%
\begin{array}{ccccc}
  N & 0 & \dots &\dots & K_0 \\
  J & N & 0 & \dots & K_1 \\
  \vdots & \ddots & \ddots & \ddots & \vdots \\
  0 & \dots & J & N & K_{m-2} \\
  0 & \dots & \dots & J & N+K_{m-1} \\
\end{array}%
\right),$$ for suitable $k_0,\dots,k_{m-1}\in F$. It is easy to
verify that
$$
[M,N]=\al\cdot I,
$$
and
$$
M J=0=JM,\; M K_i=0=K_i M.
$$
This shows that $R$ is a faithful representation, where the
minimal polynomials of $R(x_1)$ and $R(y_1)$ are as stated.

(2) Let $K$ be an algebraic closure of $F$. We have
$f(X^p)=g(X)^p$ for some $g\in K[X]$. Let $\ga_1,\dots,\ga_m$ be
the roots of $g(X)$ in $K[X]$ (these will be distinct if $F$ is
perfect, but not necessarily so in general).

Let $S:\h(1,K)\to\gl(pm,K)$ be the extension of $R$ to $\h(1,K)$.
Let $V$ and $W$ stand for the column spaces $F^{pm}$ and $K^{pm}$,
respectively.

Since $z$ and $x_1$ act on $W$ with single eigenvalues $\al$ and
$\be$, while $y_1$ acts with eigenvalues $\ga_1,\dots,\ga_m$ on
$W$, it follows from Theorem~\ref{rephn} that the
$\h(1,K)$-module~$W$ has $m$ composition factors, each of which is
isomorphic to some $V_{\al,\be,\ga_i}$.

Suppose, if possible, that $V$ is a reducible $\h(1,F)$-module. Then
there is a non-zero proper $\h(1,F)$-submodule $U$ of $V$. Then
$U\otimes K$ is an $\h(1,K)$-submodule of $V\otimes K\cong W$, so it has $k$
composition factors $V_{\al,\be,\ga_i}$ for some $1\leq k<m$.
Let $q(X)\in F[X]$ be the
characteristic polynomial of $y_1$ acting on $U$. Then $q(X)$ is the
characteristic polynomial of $y_1$ acting on $U\otimes K$, so
$$
q(X)=\underset{i\in I}\Pi(X-\ga_i)^p=\underset{i\in I}\Pi
(X^p-\ga_i^p),
$$
where $I$ is subset of $\{1,\dots,m\}$ of size $k$. In particular,
$q(X)=h(X^p)$, where $h\in K[X]$ has degree $k$. Since $q(X)\in
F[X]$, we infer $h(X)\in F[X]$. On the other hand, since $U$ is a
$y_1$-invariant subspace of $V$, it follows that $q(X)$ is a
factor of the characteristic polynomial of $y_1$ acting on $V$,
namely $f(X^p)$. Then $f(X^p)=h(X^p)a(X)$ for some $a(X)\in F[X]$.
This forces $a(X)=b(X^p)$, where $b(X)\in F[X]$. Thus,
$f(X^p)=h(X^p)b(X^p)$, which implies $f(X)=h(X)b(X)$. Since $h(X)$
has degree $k$, where $1\leq k<m$, the irreducibility of $f$ is
contradicted. This proves that $V$ is an irreducible
$\h(1,F)$-module.

(3) Suppose $f(X)=(X-\ga^p)^m$ for some $\ga\in F$. Then
$f(X^p)=(X-\ga)^{pm}$. This is a power of an irreducible
polynomial, so the column space $V=F^{pm}$ is uniserial as a
module under $y$ and hence as $\h(1)$-module. Since $\dim V=pm$,
$z$ acts via $\alpha\neq 0$ on $V$, $x_1$ has the only eigenvalue
$\be$ on $V$ and $y_1$ only $\ga$, it follows from
Theorem~\ref{rephn} that $V$ has $m$ composition factors, all
isomorphic to $V_{\al,\be,\ga}$.
\end{proof}

\begin{note}\label{imp}{\rm  It is clear that if $m>1$ then the faithful irreducible
$\hn$-module obtained in case (2) of Theorem \ref{cuando}
is inequivalent to any other previously discussed in the paper.

Suppose $F$ is perfect. Then $f(X^p)=g(X)^p$, where $g(X)\in F[X]$
is irreducible. Let $\ga_1,\dots,\ga_m$ be the distinct roots of
$g$ in $K$. Let $W(\ga_i)$ be the generalized eigenspace of $y_1$
acting on $W$. Since $y_1$ has minimal polynomial $g(X)^p$ acting
on~$W$, we see that
$$
W=W(\ga_1)\oplus\cdots\oplus W(\ga_m).
$$
Here each $W(\ga_i)$ is a faithful $\h(1,K)$-submodule of $W$,
where $x_1,y_1,z$ act with eigenvalues $\be$, $\ga_i,\al$,
respectively. Since $\dim(W)=pm$, it follows from
Theorem~\ref{rephn} (or Corollary \ref{co1}) that $W(\ga_i)\cong
V_{\al,\be,\ga_i}$. Thus, while $V$ is an irreducible $\h(1,F)$-module, its extension $W\cong V\otimes K$
is isomorphic, as $\h(1,K)$-module, to the direct sum of the $m$ non-isomorphic $\h(1,K)$-modules $V_{\al,\be,\ga_i}$, $1\leq i\leq m$.
}
\end{note}

\end{document}